\documentclass[11pt,reqno]{amsart}
\usepackage[T1]{fontenc}
\usepackage{lmodern}
\usepackage{amsmath,amssymb,amsthm,mathtools,mathrsfs}
\usepackage[margin=1in]{geometry}
\usepackage{booktabs,array,tabularx}
\usepackage{float}
\usepackage[expansion=false]{microtype}
\usepackage[colorlinks=true,linkcolor=blue,citecolor=blue,urlcolor=blue]{hyperref}
\hypersetup{pdftitle={Two-term dilogarithms, mixed-base Nahm sums, and Fricke symmetry},
 pdfauthor={Cetin Hakimoglu-Brown},
 pdfsubject={Rogers dilogarithms, Nahm sums, and modular product transformations}}
\newtheorem{theorem}{Theorem}[section]
\newtheorem{lemma}[theorem]{Lemma}
\newtheorem{proposition}[theorem]{Proposition}
\newtheorem{corollary}[theorem]{Corollary}
\newtheorem{conjecture}[theorem]{Conjecture}
\theoremstyle{remark}

\numberwithin{equation}{section}
\newcommand{\HH}{\mathbb H}
\newcommand{\QQ}{\mathbb Q}
\newcommand{\ZZ}{\mathbb Z}

\newcommand{\RR}{\mathbb R}

\DeclareMathOperator{\Li}{Li}

\DeclareMathOperator{\FT}{FT}
\title[Two-term dilogarithms and modular products]{Two-term dilogarithms, mixed-base Nahm sums, and Fricke symmetry}
\author{Cetin Hakimoglu-Brown}
\date{}
\begin{document}
\begin{abstract}
We study the arithmetic constraints that two-term Rogers dilogarithm
relations impose on rank-two Nahm systems with mixed denominator steps.
Monomial complement equations give a uniform conversion to quadratic
data in every algebraic degree. Within a specified ten-record
simplest-cubic family, positive coupling, positive definiteness, and
integer-valued exponents leave, for $m>2$, exactly the steps $(1,3)$,
$(1,13)$, and $(1,31)$. At the smallest compatible modulus $31$, the
required support cannot be a multiplicative subgroup coset. Quadratic
examples include known product identities and two systems for which
the first radial correction excludes every individual modular sum with
rational linear terms. For the index-$13$ system, we prove a two-term
Rogers identity and recover its arguments from the matrix of a
previously proved two-dimensional Fricke transformation. Its cusp
growth agrees with the resulting saddle action, and an independent
calculation reproduces all three leading Fricke amplitudes exactly. The associated
two-summand product evaluations remain conjectural; the small rank and
dimension contrast with the classical rank-five, six-component
Andrews--Gordon construction at the same modulus.  A mixed-signature cubic family illustrates the limits
of the arithmetic restriction.
\end{abstract}
\maketitle

\section{Introduction}\label{sec:intro}
The classical Andrews--Gordon identities extend the Rogers--Ramanujan
identities to every odd modulus $M=2k+1\geq5$ \cite{Andrews}. Their standard
sum representations have rank $k-1$, and their normalized modular vector
has $k$ components \cite[Section~1.3]{Ono}. Thus the rank and dimension
grow as $(M-3)/2$ and $(M-1)/2$: at modulus $13$ they are five and six.
The central example here has a two-dimensional Fricke-stable product
space and an associated rank-two saddle system. Its proposed product
evaluations use two summands, each with only two summation variables.
The product transformation was proved independently of those conjectures
in \cite{FrickeCompanion}.
This small rank and dimension at level $13$ motivate the connection
between the modular geometry and the cubic dilogarithm arithmetic.

For a rank-two Nahm sum, the leading radial growth is a weighted sum of
two Rogers dilogarithms. We use such relations to restrict the algebraic
complement equations before testing quadratic forms and linear terms. Put
$L(x)=\operatorname{Li}_2(x)+\tfrac12\log x\log(1-x)$ for $0<x<1$.
A relation $mL(x)+L(y)\in\QQ\pi^2$ is useful input, but its conversion
to a Nahm matrix also requires explicit monomial complement equations.
The conversion is independent of algebraic degree. Our rational and
quadratic comparisons include both known product evaluations and two
systems for which the first radial correction excludes every individual
modular sum with rational linear terms
(Theorem~\ref{thm:quadratic-obstruction}).

Our main arithmetic restriction is Theorem~\ref{thm:three-matrix}.
Within the ten normalized simplest-cubic records of the companion
classification \cite{CubicClassification}, the requirements of positive
coupling, positive definiteness, and integral exponents leave, for $m>2$,
only the indices $3$, $13$, and $31$. This is an exact finite-family
statement. The companion manuscript leaves a second totally real cubic
family in a finite unverified region, and no higher-degree classification
is assumed here. The denominator index, the conductor of the saddle
field, and the modulus of a possible product must be kept distinct.

The index-$13$ row is especially compatible with a modular product:
its density requires four supported residues at modulus $13$, and the
cubic-residue cosets supply exactly three supports of that size. The
Fricke theorem of \cite{FrickeCompanion} establishes a two-dimensional
transformation law for the resulting products. We identify its matrix
coefficients with the dilogarithmic saddle and compare the radial
actions. This explains a concrete structural compatibility at $13$
without claiming that it is the last possible modulus.

We recall the Fricke theorem from \cite[Theorem~1.1]{FrickeCompanion}
and refer there for its complete proof. The present paper develops the
associated rank-two Nahm system, the arithmetic sieve, the saddle
action, and the quadratic obstructions. All proofs of these arithmetic
results are included below.

Put $q=e^{2\pi i\tau}$ for $\tau\in\HH$, and write
\[
 (a_1,\ldots,a_k;q)_\infty
 =\prod_{i=1}^k\prod_{n\geq0}(1-a_iq^n).
\]
All fractional powers mean $q^r=e^{2\pi i r\tau}$. The cubic residues modulo
$13$ and their two cosets are
\begin{equation}\label{eq:cosets}
 C_0=\{\pm1,\pm5\},\qquad C_1=\{\pm2,\pm3\},\qquad
 C_2=\{\pm4,\pm6\}.
\end{equation}
Define
\begin{equation}\label{eq:products}
\begin{aligned}
 P_0(q)&=(q,q^5,q^8,q^{12};q^{13})_\infty^{-1},\\
 P_1(q)&=(q^2,q^3,q^{10},q^{11};q^{13})_\infty^{-1},\\
 P_2(q)&=(q^4,q^6,q^7,q^9;q^{13})_\infty^{-1}.
\end{aligned}
\end{equation}
For $\zeta=e^{2\pi i/13}$, let
\begin{equation}\label{eq:perioddefs}
 \eta_j=\sum_{a\in C_j}\zeta^a,\qquad
 \alpha=\eta_1-\eta_2,\quad \beta=\eta_0-\eta_2,\quad
 \gamma=\eta_1-\eta_0.
\end{equation}
We will use the product relations \cite[Lemma~2.1]{FrickeCompanion}
\begin{equation}\label{eq:theta}
 P_0=P_1+qP_2,\qquad
 P_0P_1P_2=\frac{(q^{13};q^{13})_\infty}{(q;q)_\infty}.
\end{equation}
The first reduces the three normalized products to two components.

\begin{theorem}[{\cite[Theorem~1.1]{FrickeCompanion}}]\label{thm:main}
For every $\tau\in\HH$, the vector
\[
 \mathbf H(\tau)=
 \begin{pmatrix}q^{-1/6}P_0(q)\\q^{-1/6}P_1(q)\end{pmatrix}
\]
satisfies
\begin{equation}\label{eq:main}
 \mathbf H\!\left(-\frac1{13\tau}\right)=S_{13}\mathbf H(\tau),
 \qquad
 S_{13}=\frac1{\sqrt{13}}
 \begin{pmatrix}\gamma&\beta\\\alpha&-\gamma\end{pmatrix}.
\end{equation}
Moreover, $S_{13}^2=I$ and
$\mathbf H(\tau+1)=e^{-\pi i/3}\mathbf H(\tau)$.
\end{theorem}

In this product basis, $\sqrt{13}\,S_{13}$ has entries in the cyclic cubic
field $K=\QQ(\eta_0)$. Its projective action is defined over $K$, whereas
the entries of $S_{13}$ generate $K(\sqrt{13})=\QQ(\zeta)^+$.
The matrix also has a sine-product expression:
\begin{equation}\label{eq:sinematrix}
 S_{13}=\frac14
 \begin{pmatrix}s_2^{-1}&s_1^{-1}\\s_0^{-1}&-s_2^{-1}\end{pmatrix},
 \qquad
 s_j=\prod_{\substack{a\in C_j\\1\leq a\leq6}}
 \sin\frac{\pi a}{13}.
\end{equation}
The equivalence of \eqref{eq:main} and \eqref{eq:sinematrix} is
\cite[Lemma~3.3]{FrickeCompanion}. The complete proof of
Theorem~\ref{thm:main}, including the scalar multiplier, is given in
\cite[Sections~2--6]{FrickeCompanion}.

\subsection{Related transformations and theta identities}\label{sec:related}
Table~\ref{tab:transforms} displays the growing ranks and dimensions of
the classical family alongside selected smaller systems. Rank counts
summation variables in the stated representation, and dimension counts
vector components. For our level-$13$ row, rank two refers to the
conjectural sum representations in Conjecture~\ref{conj:products};
the two-dimensional product transformation is proved independently.

\begin{table}[H]
\caption{Selected transformation laws. Entries are given up to signs in
the cited normalizations. Rank counts summation variables; dimension
counts vector components.}\label{tab:transforms}
\centering\small
\renewcommand{\arraystretch}{1.35}
\begin{tabularx}{\textwidth}{@{}cccl>{\raggedright\arraybackslash}X@{}}\toprule
Modulus&Rank&Dimension&Entries of $S$&Status\\\midrule
5&1&2&$\dfrac2{\sqrt5}\sin(k\pi/5)$&Classical Rogers--Ramanujan transformation; see \cite{Mizuno,WangZhang}.\\
7&2&3&$\dfrac2{\sqrt7}\sin(k\pi/7)$&Classical Andrews--Gordon system; see \cite{Andrews} and \cite[Section~1.3]{Ono}.\\
9&3&4&$\dfrac2{\sqrt9}\sin(k\pi/9)$&Classical Andrews--Gordon system; see \cite{Andrews,Ono}.\\
9&2&3&$\dfrac1{2\sqrt3\sin(k\pi/9)}$&Product law: \cite{WangZhang}; sums: three symmetric cases \cite{MizunoKR}, all five \cite{Xia}.\\
11&4&5&$\dfrac2{\sqrt{11}}\sin(k\pi/11)$&Classical Andrews--Gordon system; see \cite[Section~4]{Ono}.\\
11&3&5&$\sqrt{\dfrac2{11}}\sin(k\pi/11)$&Dual-vector formulas \cite[(54)]{Mizuno}; proved in \cite[Theorem~1.6]{WangWang3}.\\
13&5&6&$\dfrac2{\sqrt{13}}\sin(k\pi/13)$&Classical Andrews--Gordon system; see \cite{Andrews,Ono}.\\
13&2&2&$S_{13}$ in \eqref{eq:main}&Product law: Theorem~\ref{thm:main}. Sum identities: Conjecture~\ref{conj:products}.\\\bottomrule
\end{tabularx}
\end{table}

The classical rows have $(\text{rank},\text{dimension})=(1,2),(2,3),
(3,4),(4,5),(5,6)$ at moduli $5,7,9,11,13$. The Kanade--Russell row at
$9$ is already a smaller system; our level-$13$ row combines a proved
two-component Fricke law with conjectural rank-two product evaluations.
These are comparisons of the specified constructions, not lower bounds
for every representation at the same modulus. In particular, the
classical vectors transform under $\tau\mapsto-1/\tau$, while the
level-$13$ theorem concerns $\tau\mapsto-1/(13\tau)$.

The rows of Table~\ref{tab:transforms} also use different normalizations: the
listed modulus need not be the scaling parameter $N$ in a law of the form
$\mathbf g(-1/\tau)=S\mathbf g(\tau/N)$. The classical modulus-$7$ row
uses denominator steps $(1,1)$ and three vector components. A different
rank-two construction with steps $(1,2)$ separates the first summation
index by parity, giving six components and $3\times3$ blocks with entries
$\sqrt{2/7}\sin(k\pi/7)$; see \cite[(42)--(46)]{Mizuno} and
\cite[Theorem~1.3]{WangWang2}. The rank-three modulus-$11$ example has
two five-component vectors, with
\[
 \mathbf g(-1/\tau)=S\mathbf g^\vee(\tau/2),\qquad
 \mathbf g^\vee(-1/\tau)=2S\mathbf g(\tau/2)
\]
in Mizuno's normalization \cite[(54)]{Mizuno}. The two modulus-$11$ rows
therefore describe distinct summation systems, despite their similar sine
coefficients.

The modulus-$7$ Andrews--Gordon system already has rank two and vector
dimension three. Its summation matrix is
$A_7^{AG}=\left(\begin{smallmatrix}2&2\\2&4\end{smallmatrix}\right)$
with $D=I_2$; its three-component sine transformation is the classical
case recorded in Table~\ref{tab:transforms}
\cite{Andrews}\cite[Section~1.3]{Ono}.

A closer precedent for mixed denominator steps is the modulus-nine
transformation arising in Mizuno's study of Nahm sums for symmetrizable
matrices \cite{Mizuno}. Put
\[
 B_a(q)=(q^a,q^3,q^6,q^{9-a};q^9)_\infty^{-1}
 \quad(a=1,2,4),\qquad
 \mathbf K_9(\tau)=
 \begin{pmatrix}q^{-1/18}B_1(q)\\q^{5/18}B_2(q)\\q^{11/18}B_4(q)\end{pmatrix}.
\]
Wang and Zhang \cite[Theorem~1.1]{WangZhang} prove the product-side formula
\begin{equation}\label{eq:mod9}
 \mathbf K_9(-1/\tau)=
 \begin{pmatrix}b_1&b_2&b_4\\b_2&-b_4&-b_1\\b_4&-b_1&b_2\end{pmatrix}
 \mathbf K_9(\tau/3),\qquad
 b_k=\frac1{2\sqrt3\sin(k\pi/9)}.
\end{equation}
Their proof transforms generalized eta factors into twisted products and
then uses Garvan's $3$-dissection formulas to re-expand them.
They also record Mizuno's independent proof using the $A_2$ Macdonald
identity and binary theta series \cite[Remark~2]{WangZhang}.
The three symmetric Kanade--Russell sum--product identities are now
proved by Mizuno \cite[Theorem~1.1]{MizunoKR} and, independently, by
Xia \cite[Theorem~1.1]{Xia}. Thus \eqref{eq:mod9} also applies to the
corresponding normalized double sums. Xia proves all five modulo-nine
identities, together with four individual dual identities and a fifth
dual companion. These results appeared in September 2026 preprints; their independence
is recorded in \cite[Introduction, version~2]{MizunoKR}.
Further transformation laws for generalized rank-two and rank-three
Nahm-sum vectors appear in \cite{WangWang2,WangWang3}.

At modulus $13$, the classical Andrews--Gordon products \cite{Andrews}
also occur in Jain's series identities \cite[(3.15)--(3.20)]{Jain}.
Writing
\[
 J_j(q)=\frac{(q^j,q^{13-j},q^{13};q^{13})_\infty}{(q;q)_\infty}
 \quad(1\leq j\leq6),\qquad
 E(q)=\frac{(q^{13};q^{13})_\infty}{(q;q)_\infty},
\]
direct cancellation of product factors gives
\begin{equation}\label{eq:classical13-products}
 P_0=\frac{E^2}{J_1J_5},\qquad
 P_1=\frac{E^2}{J_2J_3},\qquad
 P_2=\frac{E^2}{J_4J_6}.
\end{equation}
Thus the four-residue products $P_j$ are quotients of the familiar
ten-residue products $J_j$ and a common factor. Further modulus-$13$
comparisons include the cylindric Kanade--Russell identities studied by
Uncu \cite{Uncu} and the tredecic theta identities of Vasuki and Prabhu
\cite{VasukiPrabhu}. The relation $P_0=P_1+qP_2$ itself is a specialization
of the Weierstrass theta identity.

Although the classical modulus-$13$ vector has six components, these
particular normalized product quotients span a two-dimensional space.
The recalled Theorem~\ref{thm:main} establishes its Fricke stability with the explicit
Gaussian-period matrix \eqref{eq:main}, and Section~\ref{sec:framework}
connects it to the rank-two saddle. The elementary relations
\eqref{eq:classical13-products} do not by themselves establish this linear
transformation or the conjectural sum evaluations. The Fricke proof starts
with eta transformation and identifies the resulting action through a
degree-two map on $X_1(13)$. The cosets \eqref{eq:cosets} both partition
the nonzero residues and index the Gaussian periods in the matrix,
connecting the product supports to the coefficients of the transformation.

\subsection{From modular coefficients to dilogarithmic saddle points}
The Fricke matrix raises a further arithmetic question: do its coefficients
identify the saddle point of a related $q$-series? We show that ratios of
its entries recover two real cubic numbers satisfying a weighted Rogers
dilogarithm identity. Their complement equations specify a symmetrizable
rank-two quadratic form, and the dilogarithm evaluation determines its
leading radial exponential. Independently, the Fricke transformation gives
the same exponential from the product expansion at the transformed cusp.
Proposition~\ref{prop:amplitudes} also identifies all three leading
amplitudes by an independent saddle calculation.

This connection extends a familiar theme in the
Rogers--Ramanujan--Andrews--Gordon setting \cite{Andrews,Kirillov} to a
symmetrizable system with denominator steps $1$ and $13$
\cite{Kanade,Mizuno}. We propose three two-summand representations of the
products, prove their exact linear dependence, and verify their coefficients
through degree $800$. These infinite sum--product identities remain
conjectural. The Fricke theorem of \cite{FrickeCompanion} and the dilogarithm
evaluation proved here are independent of these conjectures.

Section~\ref{sec:cubic} identifies the cubic parameters and states
their Rogers evaluation, proved in Appendix~\ref{app:certificate}.
Section~\ref{sec:framework} gives the conversion in arbitrary degree
and the asymptotic conditions. Section~\ref{sec:examples} compares
rational, quadratic and cubic instances, including exact obstructions
for two quadratic systems. Section~\ref{sec:sieve} then restricts the
input to a finite cubic family and proves the three-matrix theorem and
the coset obstruction. Section~\ref{sec:conjectures} states the
sum--product conjectures, their finite evidence, and their exact
leading amplitudes.
Section~\ref{sec:signature} addresses mixed signatures.

\section{Cubic parameters and a Rogers dilogarithm identity}\label{sec:cubic}
The coefficients of the Fricke matrix now supply the algebraic parameters
for a dilogarithm evaluation. We first describe these parameters by a
cubic polynomial, then identify them with ratios of matrix entries.

\subsection{The cubic parameters}
Let
\begin{equation}\label{eq:rho}
 f(X)=X^3-4X^2+X+1,\qquad \rho\in(0,1),\quad f(\rho)=0,
 \qquad w=1-\rho,\quad \sigma=\frac{w}{\rho^4}.
\end{equation}
The polynomial has exactly one root in $(0,1)$: its derivative has one zero
there, and it first increases from $f(0)=1$ and then decreases to $f(1)=-1$.
A rational isolating interval is
\begin{equation}\label{eq:rhointerval}
 \frac{726109445035}{10^{12}}<\rho<\frac{726109445036}{10^{12}}.
\end{equation}
The endpoint signs are opposite and $f'$ is negative on this interval.

\begin{lemma}\label{lem:algebra}
The element $\sigma$ lies in $(0,1)$, satisfies
$\sigma^3-69\sigma^2+66\sigma+1=0$, and obeys
\begin{equation}\label{eq:complements}
 1-\rho=\rho^4\sigma,\qquad 1-\sigma=\rho^{13}\sigma^4.
\end{equation}
In the basis $1,\rho,\rho^2$ one has
\begin{equation}\label{eq:reduction}
 \rho^4=15\rho^2-5\rho-4,\qquad \sigma=41-66\rho+15\rho^2.
\end{equation}
\end{lemma}
\begin{proof}
Reduction by $f(\rho)=0$ gives \eqref{eq:reduction} and the polynomial for
$\sigma$. The first equation in \eqref{eq:complements} is its definition.
The polynomial identity
\[
 X^4+X-1-X(1-X)^4=-(X^2-X+1)f(X)
\]
gives $\rho^4+\rho-1=\rho(1-\rho)^4$.
Division by $\rho^4$ proves the second complement equation. The definition
gives $\sigma>0$, and that equation then gives $\sigma<1$.
\end{proof}
Numerically, $\rho=0.7261094450\ldots$ and
$\sigma=0.9853005201\ldots$; these decimals are not used in the proof.

The period calculation in \cite[Lemma~3.1]{FrickeCompanion} gives
the polynomial $X^3+X^2-4X+1$ and
\begin{equation}\label{eq:periodrelations}
 \eta_1=2-2\eta_0-\eta_0^2,\qquad
 \eta_2=-3+\eta_0+\eta_0^2.
\end{equation}
Also
\[
 \eta_0=2\bigl(\cos(2\pi/13)-\cos(3\pi/13)\bigr),\qquad
 0<\eta_0<\frac{2\pi}{13}<1.
\]
The upper bound is the mean value theorem. Substituting $X=1-Y$ in the
period polynomial gives $-f(Y)$, so
\begin{equation}\label{eq:rhoperiod}
 \rho=1-\eta_0.
\end{equation}
Both cubics are irreducible over $\QQ$ by the rational-root test. Thus
$K=\QQ(\rho)=\QQ(\eta_0)$ is the fixed field of the cubic-residue subgroup
in $\QQ(\zeta)$: it is the cyclic cubic field of conductor $13$, a proper
subfield of $\QQ(\zeta)^+$.

\subsection{Recovery from the Fricke matrix}
The same parameters occur in the transformation coefficients:
\begin{equation}\label{eq:matrixratios}
 \rho=\frac\beta\alpha,\qquad
 1-\rho=\frac\gamma\alpha,\qquad
 \sigma=\frac{\gamma\alpha^3}{\beta^4}.
\end{equation}
Equivalently, in terms of the entries of $S_{13}$,
\[
 \rho=\frac{S_{12}}{S_{21}},\qquad
 \sigma=\frac{S_{11}S_{21}^3}{S_{12}^4}.
\]

Put $e=\eta_0$. The period identities give
\[
 \alpha=5-3e-2e^2,\qquad \beta=3-e^2,\qquad \gamma=2-3e-e^2.
\]
Reduction by $e^3=-e^2+4e-1$ gives $\alpha e=\gamma$.
Since $\rho=1-e$ and $\alpha=\beta+\gamma$, the first two equations in
\eqref{eq:matrixratios} follow. The third is
$\sigma=(1-\rho)/\rho^4$. The common factor $1/\sqrt{13}$ cancels in
the displayed ratios of matrix entries.

\subsection{The Rogers evaluation}
The complement equations in Lemma~\ref{lem:algebra} give a multiplicative
relation between the cubic parameters. Their additive relation is expressed
by the real Rogers dilogarithm. To state it, write
\begin{equation}\label{eq:Rogersdefinition}
 L(x)=\Li_2(x)+\frac12\log x\log(1-x)\quad(0<x<1),\qquad
 L(1)=\frac{\pi^2}{6},
\end{equation}
where $\Li_2(x)=\sum_{n\geq1}x^n/n^2$.

\begin{theorem}\label{thm:dilog}
For the arguments in \eqref{eq:rho},
\begin{equation}\label{eq:dilog}
 13L(\rho)+L(\sigma)=\frac{5\pi^2}{3}.
\end{equation}
\end{theorem}
The proof is given in Appendix~\ref{app:certificate}. It combines fourteen
real five-term relations with eight reflection relations, with all arguments
in $(0,1)$.

Together, the weighted identity and the complement equations determine
the rank-two system considered next.

\section{Weighted Rogers relations and symmetrizable Nahm systems}
\label{sec:framework}
We now place the level-$13$ evaluation in a framework that applies in
every algebraic degree. Complement equations determine the matrix,
the Rogers value determines the leading radial action, and higher
coefficients impose further conditions on modularity. We give the
common conversion first, then apply it to the parameters of
Section~\ref{sec:cubic} and to the examples in
Section~\ref{sec:examples}.
\subsection{Complement equations and quadratic forms}
We use Mizuno's convention \cite{Mizuno}. Let
$D=\operatorname{diag}(d_1,\ldots,d_N)$,
where $d_i$ are positive integers, and suppose that $AD$ is symmetric
positive definite. A symmetrizable Nahm sum has the form
\begin{equation}\label{eq:nahmsum}
 f_{A,\mathbf b,c,\mathbf d}(q)=
 \sum_{\mathbf n\in\ZZ_{\geq0}^N}
 \frac{q^{\frac12\mathbf n^TAD\mathbf n+\mathbf b^T\mathbf n+c}}
 {\prod_{i=1}^N(q^{d_i};q^{d_i})_{n_i}}.
\end{equation}
Here $(a;q)_n=\prod_{j=0}^{n-1}(1-aq^j)$. The positive Nahm solution
satisfies
\begin{equation}\label{eq:nahmequations}
 1-z_i=\prod_j z_j^{A_{ij}}.
\end{equation}

Suppose a two-term relation $mL(z_0)+L(z_1)=c_0\pi^2$, with
$0<z_0,z_1<1$ and $m\in\ZZ_{>1}$, is accompanied by
\begin{equation}\label{eq:generalcomplements}
 1-z_0=z_0^a z_1^b,\qquad 1-z_1=z_0^{mb}z_1^d.
\end{equation}
Rational powers are interpreted in the positive real embedding. These
additional equations give
\begin{equation}\label{eq:generalmatrix}
 A=\begin{pmatrix}a&b\\mb&d\end{pmatrix},\qquad
 D=\operatorname{diag}(1,m),\qquad
 AD=\begin{pmatrix}a&mb\\mb&md\end{pmatrix}.
\end{equation}
The last matrix is positive definite precisely when $a>0$ and
$ad-mb^2>0$. Its quadratic form is
\[
 \frac a2 r^2+mb\,rs+\frac{md}{2}s^2.
\]
The coefficient $m$ in the Rogers relation is therefore compatible with the
second denominator step. The Rogers value alone does not specify the
complement equations, the exponent matrix, or the linear terms.

There is a corresponding exterior-product cancellation. In the rationalized
exterior square of the multiplicative group of a field containing the
arguments, clearing denominators in \eqref{eq:generalcomplements} gives
\[
 z_0\wedge(1-z_0)+\frac1m z_1\wedge(1-z_1)
 =b\,z_0\wedge z_1+b\,z_1\wedge z_0=0.
\]
This explains the weighting, but does not on its own determine the Rogers
value in the chosen real embedding.

\subsection{A common conversion in every algebraic degree}
\label{sec:conversion}
The following construction uses no assumption on the degree of the field.
Let $0<x,y<1$, let $m>1$ be an integer, and suppose that $x,1-x$
are multiplicatively independent and that rational exponents satisfy
\begin{equation}\label{eq:monomial-input}
 y=x^p(1-x)^h,\qquad 1-y=x^u(1-x)^v,
 \qquad pv-hu=-m.
\end{equation}
These are the monomial data that accompany the Rogers value. Rational,
quadratic and cubic examples all pass through the same conversion.
Field degree becomes relevant later, when an arithmetic classification
is used to supply or restrict the input records.

\begin{proposition}\label{prop:monomial-matrix}
Suppose \eqref{eq:monomial-input} holds with $h\ne0$. Its unique rational
Nahm matrix for the ordered pair $(x,y)$ is
\begin{equation}\label{eq:converted-matrix}
 A=\frac1h\begin{pmatrix}-p&1\\m&v\end{pmatrix},\qquad
 D=\operatorname{diag}(1,m),\qquad
 \det(AD)=-\frac{mu}{h}.
\end{equation}
The cross term is positive and $AD$ is positive definite exactly when
$h>0$, $p<0$, and $u<0$. A rational linear term can make
$\tfrac12(r,s)AD(r,s)^T$ integer-valued on $\ZZ^2$ exactly when
all three entries of the symmetric matrix $AD$ are integers.
\end{proposition}
\begin{proof}
Solving the first equation of \eqref{eq:monomial-input} for $1-x$ gives
$1-x=x^{-p/h}y^{1/h}$. Substitution into the second gives
$1-y=x^{u-pv/h}y^{v/h}=x^{m/h}y^{v/h}$. Independence of $x,1-x$
and $h\ne0$ imply independence of $x,y$, proving uniqueness of the
exponents. Computing the determinant gives \eqref{eq:converted-matrix};
Sylvester's criterion and the sign of $m/h$ prove the positivity statement.
For integrality, second forward differences in $r$ and $s$ give the two
diagonal entries, and a mixed forward difference gives the off-diagonal
entry. These must be integers. Conversely, for an integral symmetric
matrix $B$, subtracting $B_{11}r/2+B_{22}s/2$ from
$\tfrac12(r,s)B(r,s)^T$ leaves
$B_{11}\binom r2+B_{12}rs+B_{22}\binom s2$, an integer.
\end{proof}

Simultaneous reflection sends $(x,y)$ to $(1-x,1-y)$ and the
exponent record $(p,h,u,v)$ to $(v,u,h,p)$. If $A$ is invertible,
it replaces $A$ by $A^{-1}$: take logarithms of the complement
equations and interchange their two sides. Positive definiteness is
preserved, while the off-diagonal sign is reversed.

The conversion separates three questions: whether the monomial data
give a positive-definite quadratic form, whether the desired summation
lattice permits integral exponents, and whether suitable linear terms
give a modular series. Integrality is a restriction of the ansatz used
in our cubic sieve, not a general prerequisite for modularity.
The higher asymptotic conditions test the last question independently.
If $x,1-x$ are dependent, the uniqueness assertion does not apply;
the Lewin family in Section~\ref{sec:signature} illustrates that case.

\subsection{The radial action}\label{sec:radial-action}
For $q=e^{-\varepsilon}$, the leading exponential constant in the
asymptotics of \eqref{eq:nahmsum} is \cite[Section~2]{Mizuno}
\begin{equation}\label{eq:lambda}
 \Lambda=\sum_i\frac{L(1)-L(z_i)}{d_i}.
\end{equation}
Mizuno's Rogers function is shifted by $-\pi^2/6$ relative to ours.
To see the complement equations in this calculation, put $z_i=e^{-t_i}$
and scale $n_i\sim t_i/(d_i\varepsilon)$. The leading action is
\[
 \mathcal S(\mathbf t)=
 \sum_i\frac{\pi^2/6-\Li_2(e^{-t_i})}{d_i}
 -\frac12\mathbf t^TD^{-1}A\mathbf t.
\]
The matrix $D^{-1}A$ is symmetric positive definite because it equals
$D^{-1}(AD)D^{-1}$. Differentiation gives
$\log(1-e^{-t_i})=-(A\mathbf t)_i$, which is
\eqref{eq:nahmequations}. Substituting at the stationary point gives
\eqref{eq:lambda}. The full asymptotic theorem justifies this saddle-point
calculation.

In the two-variable situation above, Euler reflection gives
\begin{equation}\label{eq:weightedlambda}
 \Lambda=L(1-z_0)+\frac1mL(1-z_1)
 =\frac{\pi^2}{m}\left(\frac{m+1}{6}-c_0\right).
\end{equation}
For a reciprocal product on a set $S$ of nonzero residues modulo $M$,
\[
 P_S(q)=\prod_{\substack{n>0\\n\bmod M\in S}}(1-q^n)^{-1},
\]
the leading constant is $|S|\pi^2/(6M)$. Indeed, the Riemann sum on each
arithmetic progression gives the integral
$\int_0^\infty-\log(1-e^{-x})\,dx=\pi^2/6$.
Matching this constant restricts the density of a candidate product; it
leaves the actual residue classes undetermined.

\subsection{The level-13 system and its modular cusp expansion}
The complement equations give
\begin{equation}\label{eq:cubicmatrix}
 A=\begin{pmatrix}4&1\\13&4\end{pmatrix},\quad
 D=\operatorname{diag}(1,13),\quad
 AD=\begin{pmatrix}4&13\\13&52\end{pmatrix},\quad \det(AD)=39>0.
\end{equation}
Thus $(\rho,\sigma)$ is its positive Nahm solution. Equations
\eqref{eq:lambda} and \eqref{eq:dilog} yield
\begin{equation}\label{eq:cubiclambda}
 \Lambda=\frac{14}{13}\frac{\pi^2}{6}
 -\frac1{13}\frac{5\pi^2}{3}
 =\frac{2\pi^2}{39},\qquad \frac{6\Lambda}{\pi^2}=\frac4{13}.
\end{equation}
Within a modulus-$13$ reciprocal-product ansatz, the support therefore has
size four. There are fifteen symmetric four-element supports obtained by
choosing two of the six pairs $\{\pm a\}$; this argument does not distinguish
the three cosets \eqref{eq:cosets} from the other twelve supports.

\begin{proposition}\label{prop:productasy}
The Fricke law gives the radial asymptotics
\begin{equation}\label{eq:productasy}
 \binom{P_0(e^{-\varepsilon})}{P_1(e^{-\varepsilon})}
 =\frac1{\sqrt{13}}
 \exp\!\left(\frac{2\pi^2}{39\varepsilon}-\frac\varepsilon6\right)
 \left\{\binom\alpha\beta+O\!\left(e^{-4\pi^2/(13\varepsilon)}\right)\right\}
 \quad(\varepsilon\downarrow0).
\end{equation}
The third product satisfies
$P_2(e^{-\varepsilon})\sim\gamma e^{\Lambda/\varepsilon}/\sqrt{13}$,
with $\Lambda=2\pi^2/39$.
\end{proposition}
\begin{proof}
Set $\tau=i\varepsilon/(2\pi)$ and
$Q=e^{2\pi iW\tau}=e^{-4\pi^2/(13\varepsilon)}$.
The product definitions imply
\[
 \mathbf H(W\tau)=Q^{-1/6}\left\{\binom11+O(Q)\right\}.
\]
Apply $\mathbf H(\tau)=S_{13}\mathbf H(W\tau)$ and use
$S_{13}(1,1)^T=(\alpha,\beta)^T/\sqrt{13}$. Removing
$q^{-1/6}=e^{\varepsilon/6}$ proves \eqref{eq:productasy}.
The last assertion follows from $qP_2=P_0-P_1$ in
\eqref{eq:theta} and $\alpha-\beta=\gamma>0$.
\end{proof}
Thus the finite Rogers certificate and the Fricke cusp expansion give the
same exponential $\Lambda=2\pi^2/39$. They also give the same algebraic
arguments through \eqref{eq:matrixratios}. This compatibility does not
identify the independently defined sums with the products; that is still
Conjecture~\ref{conj:products}. If the conjecture holds, its normalized
sum vector inherits Theorem~\ref{thm:main} immediately.

\subsection{Higher conditions for two summands}\label{sec:higher}
Linear terms do not change $\Lambda$, but enter the higher asymptotic
conditions \cite{Kanade,Mizuno,VlasenkoZwegers}. Suppose
\[
 F_i(e^{-\varepsilon})=K_i e^{\Lambda/\varepsilon}
 \exp\bigl(\beta_i\varepsilon+\gamma_i\varepsilon^2+
 \delta_i\varepsilon^3+O(\varepsilon^4)\bigr),\qquad K_i>0.
\]
For $H(q)=F_1(q)+q^kF_2(q)$, put
\[
 p=\frac{K_1}{K_1+K_2},\qquad p'=1-p,\qquad
 \Delta=\beta_2-k-\beta_1.
\]
Write
\[
 \log H(e^{-\varepsilon})=
 \frac{\Lambda}{\varepsilon}+\log(K_1+K_2)
 +A_1\varepsilon+A_2\varepsilon^2+A_3\varepsilon^3+O(\varepsilon^4).
\]
\begin{proposition}\label{prop:higher}
Under these assumptions,
\begin{align}
 A_2&=p\gamma_1+p'\gamma_2+\frac12pp'\Delta^2,
 \label{eq:secondasy}\\
 A_3&=p\delta_1+p'\delta_2+pp'\Delta(\gamma_2-\gamma_1)
 +\frac16pp'(p-p')\Delta^3.\label{eq:thirdasy}
\end{align}
\end{proposition}
\begin{proof}
After factoring out the first exponential, expand
\[
 \log\!\left(p+p'\exp\bigl(\Delta\varepsilon+
 (\gamma_2-\gamma_1)\varepsilon^2+
 (\delta_2-\delta_1)\varepsilon^3+O(\varepsilon^4)\bigr)\right).
\]
Adding the exponent belonging to $F_1$ and collecting terms gives the
formulas.
\end{proof}
A weight-zero modular function meromorphic at the relevant cusp has only
exponentially small corrections after its leading cusp term. Up to a
$q$-normalization, this requires $A_2=A_3=\cdots=0$. The amplitude ratio
and relative shift enter these equations, so a failed modularity test for an
individual summand need not rule out a modular sum of two constituents.
We do not claim exact vanishing of these higher coefficients for the
conjectures below.

\section{Rational, quadratic and cubic instances of the construction}
\label{sec:examples}
We compare successful and obstructed inputs to the same construction,
ordered by algebraic degree. Each row supplies a Rogers relation and
explicit monomial data; the matrix, the radial action and the modularity
test are then separate steps. In particular, the quadratic field alone
does not predict success: the two examples over $\QQ(\sqrt2)$ below
have different matrices and different outcomes.

Table~\ref{tab:rogers} is a comparison, not a classification. Here
$K_9=\QQ(e^{2\pi i/9})^+$ and $K=\QQ(\rho)$. The column $m$ denotes
the second denominator step, not the product modulus. Complementation
is used when needed to obtain positive cross terms. The symbols $x,y$
in the four additional quadratic rows are defined in
\eqref{eq:quadratic-inputs}; the cubic conductor-nine parameters are
given in \eqref{eq:kanadedilog}.
\begin{table}[htbp]
\caption{The same rank-two construction across algebraic degrees.
Each relation has the form $mL(x)+L(y)=c_0\pi^2$.
The outcome concerns the specified saddle system.}
\label{tab:rogers}
\centering\small
\renewcommand{\arraystretch}{1.22}
\begin{tabularx}{\textwidth}{@{}clc>{\raggedright\arraybackslash}p{0.34\textwidth}>{\raggedright\arraybackslash}X@{}}\toprule
Degree&Field&$m$&Rogers relation&Outcome\\\midrule
1&$\QQ$&3&$3L(1/4)+L(1/9)=\pi^2/6$&Capparelli products; proved after complementation\\
2&$\QQ(\sqrt2)$&4&$4L(1/\sqrt2)+L(2(\sqrt2-1))=7\pi^2/12$&G\"ollnitz--Gordon products; proved expressions\\
2&$\QQ(\sqrt5)$&2&$2L(x)+L(y)=2\pi^2/15$&Rogers--Ramanujan products; proved after complementation\\
2&$\QQ(\sqrt3)$&8&$8L(x)+L(y)=13\pi^2/12$&Individual modular sums excluded by Theorem~\ref{thm:quadratic-obstruction}\\
2&$\QQ(\sqrt2)$&6&$6L(x)+L(y)=11\pi^2/24$&Individual modular sums excluded by Theorem~\ref{thm:quadratic-obstruction}\\
2&$\QQ(\sqrt5)$&4&$4L(x)+L(y)=\pi^2/3$&Indefinite matrix; no modular evaluation established\\
3&$K_9$&3&$3L(Q_1)+L(Q_2^3)=4\pi^2/9$&Three symmetric Kanade--Russell evaluations, proved in \cite{MizunoKR,Xia}\\
3&$K$&13&$13L(\rho)+L(\sigma)=5\pi^2/3$&Proved Fricke law; conjectural sum evaluations\\\bottomrule
\end{tabularx}
\end{table}

\subsection{The rational Capparelli example}
The identity
\begin{equation}\label{eq:cappdilog}
 3L(1/4)+L(1/9)=\frac{\pi^2}{6}
\end{equation}
appears in Kanade's analysis \cite[Section~3]{Kanade}. For completeness,
the five-term identity at $(1/2,1/2)$ and $(1/3,1/3)$ gives
\[
 L(1)=L(1/4)+2L(1/3),\qquad
 2L(1/3)=L(1/9)+2L(1/4).
\]
For $z_0=1/4$, $z_1=1/9$, the complement equations are
\[
 1-z_0=z_0z_1^{-1/2},\qquad 1-z_1=z_0^{-3/2}z_1.
\]
Thus
\[
 A_0=\begin{pmatrix}1&-1/2\\-3/2&1\end{pmatrix},\qquad
 D_0=\operatorname{diag}(1,3),\qquad \det(A_0D_0)=\frac34>0.
\]
For $x=1-z_0=3/4$ and $y=1-z_1=8/9$, inverting the logarithmic
system gives
\[
 1-x=x^4y^2,\qquad 1-y=x^6y^4,\qquad
 A_C=A_0^{-1}=\begin{pmatrix}4&2\\6&4\end{pmatrix}.
\]
The familiar positive-cross-term quadratic form is therefore
$2r^2+6rs+6s^2$. A proved identity is \cite{Kanade,KursungozCapp}
\begin{equation}\label{eq:cappsum}
 \sum_{r,s\geq0}\frac{q^{2r^2+6rs+6s^2}}{(q;q)_r(q^3;q^3)_s}
 =\frac1{(q^2,q^3,q^9,q^{10};q^{12})_\infty}.
\end{equation}
Its leading constant is
$L(1/4)+\tfrac13L(1/9)=\pi^2/18$.
For the single-summand family with added exponent $B_1r+B_2s+C$,
Kanade's higher necessary conditions select $B_1=B_2=0$ and $C=-1/24$
\cite[Section~3]{Kanade}. This illustrates why the Rogers value does not
select linear terms, even when the product identity is known.

\subsection{The quadratic G\"ollnitz--Gordon example}
The G\"ollnitz--Gordon row of Table~\ref{tab:rogers} uses
\begin{equation}\label{eq:GGdilog}
 4L(1/\sqrt2)+L\bigl(2(\sqrt2-1)\bigr)=\frac{7\pi^2}{12}.
\end{equation}
For $x=1/\sqrt2$ and $y=2(\sqrt2-1)$, direct algebra gives
$1-x=x^3y$ and $1-y=x^4y^2$. Hence
\[
 A_{GG}=\begin{pmatrix}3&1\\4&2\end{pmatrix},\quad
 D_{GG}=\operatorname{diag}(1,4),\quad
 Q_{GG}(r,s)=\frac32r^2+4rs+4s^2.
\]
An identity of Kur\c{s}ung\"oz, also proved by Wang
\cite{KursungozCapp,WangDouble}, is
\begin{equation}\label{eq:GGtwo}
 \sum_{r,s\geq0}
 \frac{q^{(3r^2+r)/2+4rs+4s^2+2s}(1+q^{2r+4s+1})}
 {(q;q)_r(q^4;q^4)_s}
 =\frac1{(q,q^5,q^6;q^8)_\infty}.
\end{equation}
This provides an established two-summand precedent. Conversely, it also
proves \eqref{eq:GGdilog} using \eqref{eq:lambda}: both positive summands
have the same leading exponential, and the product permits three residues
modulo $8$, so
\[
 L(1)-L(x)+\frac14\bigl(L(1)-L(y)\bigr)=\frac{\pi^2}{16}.
\]
Rearrangement gives the displayed Rogers value.

The same product has the single-summand expression
\cite[Theorem~3.17]{LiWang}
\begin{equation}\label{eq:GGone}
 \sum_{r,s\geq0}
 \frac{q^{(3r^2-r)/2+4rs+4s^2+2s}}{(q;q)_r(q^4;q^4)_s}
 =\frac1{(q,q^5,q^6;q^8)_\infty}.
\end{equation}
Its negative linear coefficient explains why our finite comparison
allows negative coefficients. A search confined to nonnegative linear
terms could miss an available one-summand expression.

\subsection{Four further quadratic inputs and their matrices}
\label{sec:quadratic-inputs}
To test the same construction beyond the preceding successful examples,
consider the four quadratic evaluations $mL(x)+L(y)=c_0\pi^2$ with
\begin{equation}\label{eq:quadratic-inputs}
\begin{array}{c|c|c|c}
 m&x&y&c_0\\\hline
 8&\sqrt3-1&(2+\sqrt3)/4&13/12\\[2pt]
 6&\sqrt2-1&(2-\sqrt2)/4&11/24\\[2pt]
 4&\sqrt5-2&4\sqrt5-8&1/3\\[2pt]
 2&(5-\sqrt5)/10&(7-3\sqrt5)/2&2/15
\end{array}
\end{equation}
The Rogers evaluations are comparison inputs here. The matrix
conversions and the obstruction theorem below depend only on the
displayed algebraic arguments, not on a numerical dilogarithm test.
For $w=1-x$, direct algebra gives
\begin{equation}\label{eq:quadratic-monomials}
\begin{array}{c|cc|c|c}
 m&y&1-y&A&\det(AD)\\\hline
 8&x^{-4}w&x^{-4}w^3&
 \left(\begin{smallmatrix}4&1\\8&3\end{smallmatrix}\right)&32\\[3pt]
 6&x^4w^{-3}&x^2w^{-3}&
 \left(\begin{smallmatrix}4/3&-1/3\\-2&1\end{smallmatrix}\right)&4\\[3pt]
 4&x^{-1/3}w^2&x^2&
 \left(\begin{smallmatrix}1/6&1/2\\2&0\end{smallmatrix}\right)&-4\\[3pt]
 2&x^2w^{-2}&x^{1/2}w^{-3/2}&
 \left(\begin{smallmatrix}1&-1/2\\-1&3/4\end{smallmatrix}\right)&1/2
\end{array}
\end{equation}
All fractional powers are positive real roots. The pair $x,w$ is
multiplicatively independent in each case: for $m=8,6,4$ the norms
separate a nontrivial unit from a nonunit; for $m=2$ both norms are
$1/5$, and $x/w\ne1$ excludes the remaining relation. Hence the
matrices in \eqref{eq:quadratic-monomials} are the unique ones supplied
by Proposition~\ref{prop:monomial-matrix} for the ordered arguments.

\subsection{The quadratic Rogers--Ramanujan case}
For $m=2$, complementing both arguments in
\eqref{eq:quadratic-inputs} gives
\[
 z_1=\frac{5+\sqrt5}{10},\qquad z_2=\frac{3\sqrt5-5}{2},\qquad
 A_2=\begin{pmatrix}3&2\\4&4\end{pmatrix},\quad D_2=\operatorname{diag}(1,2).
\]
Indeed $1-z_1=z_1^3z_2^2$ and $1-z_2=z_1^4z_2^4$. This matrix
occurs in Mizuno's list \cite[(49)]{Mizuno}, and gives
\begin{align}
 \sum_{r,s\geq0}\frac{q^{(3r^2-r)/2+4rs+4s^2}}
 {(q;q)_r(q^2;q^2)_s}
 &=\frac1{(q,q^4;q^5)_\infty},\label{eq:quadRR1}\\
 \sum_{r,s\geq0}\frac{q^{(3r^2+r)/2+4rs+4s^2+2s}}
 {(q;q)_r(q^2;q^2)_s}
 &=\frac1{(q^2,q^3;q^5)_\infty}.\label{eq:quadRR2}
\end{align}
For a direct reduction, Euler's identities give
\[
 \frac{(-t;q)_\infty}{(t^2;q^2)_\infty}=\frac1{(t;q)_\infty},\qquad
 \sum_{r+2s=n}\frac{q^{r(r-1)/2}}{(q;q)_r(q^2;q^2)_s}
 =\frac1{(q;q)_n}.
\]
Multiplying the second equality by $q^{n^2}$ or $q^{n^2+n}$ and
summing gives the classical Rogers--Ramanujan series. Thus this
quadratic input gives known rank-two representations of the
two-component modulus-five vector. Their leading action also proves
the $m=2$ Rogers evaluation in \eqref{eq:quadratic-inputs}.

\subsection{Quadratic obstructions: failure at different stages}
\label{sec:quadratic-obstructions}
For $m=8$, retain the original arguments; for $m=6$, complement both.
The positive-cross-term matrices and their quadratic exponents are
\begin{align}
 A_8&=\begin{pmatrix}4&1\\8&3\end{pmatrix},&
 Q_8(r,s)&=2r^2+8rs+12s^2,\label{eq:quadA8}\\
 A_6&=\begin{pmatrix}3/2&1/2\\3&2\end{pmatrix},&
 Q_6(r,s)&=\tfrac34r^2+3rs+6s^2,\label{eq:quadA6}
\end{align}
with $D_m=\operatorname{diag}(1,m)$. Both $A_mD_m$ are positive
definite. The second fails the full-lattice integrality test, since its
first diagonal entry is $3/2$. Even when fractional exponents are
allowed, however, both systems fail a necessary modularity condition.

\begin{samepage}
\begin{theorem}\label{thm:quadratic-obstruction}
For $m\in\{6,8\}$ and $A\in\{A_m,A_m^{-1}\}$, no rational
$a,b,c$ make
\[
 q^c\sum_{r,s\geq0}
 \frac{q^{\frac12(r,s)AD_m(r,s)^T+ar+bs}}
 {(q;q)_r(q^m;q^m)_s}
\]
a meromorphic modular function of weight zero for a congruence
subgroup, allowing a finite-order multiplier.
\end{theorem}
\end{samepage}
The proof is in Appendix~\ref{app:quadratic-obstruction}. The first
radial correction would have to be rational; exact computation reduces
this requirement for $A_8$ and $A_6$ to rational solutions of
\[
 X^2+3Y^2=220,\qquad X^2-21Y^2=-145,
\]
respectively. Congruences modulo five and three exclude those solutions.
The inverse matrices give the same local obstructions. The method is
the asymptotic modularity test of \cite{Mizuno,VlasenkoZwegers}.

The coefficient-$4$ row in \eqref{eq:quadratic-monomials} fails earlier:
its quadratic form $r^2/12+2rs$ has indefinite matrix, as does its
inverse. This places it outside the positive-definite construction.
It does not prove nonmodularity of every convergent series: with an
added term $ar+bs$, the original sum converges for $b>0$.
No modular evaluation for that case is asserted here.

These examples show why a two-term Rogers relation alone does not
produce a modular sum. They do not classify the corresponding fields.
Moreover, Theorem~\ref{thm:quadratic-obstruction} concerns individual
unrestricted sums; it does not exclude a modular combination of them.
For a positive combination, the first correction is an amplitude-weighted
average of the constituent corrections. Its irrational parts can cancel,
so the combined conditions of Section~\ref{sec:higher} must be tested.
This is the relevant distinction for the level-$13$ conjectures.

\subsection{The cubic conductor-nine model}
The nearest motivating symmetrizable system has
\[
 A_9=\begin{pmatrix}2&1\\3&2\end{pmatrix},\qquad
 D_9=\operatorname{diag}(1,3),\qquad
 Q_9(r,s)=r^2+3rs+3s^2.
\]
The three symmetric Kanade--Russell identities, studied in
\cite{Kanade,KursungozKR,Mizuno} and now proved in
\cite[Theorem~1.1]{MizunoKR} and \cite[Theorem~1.1]{Xia},
use linear terms $(0,0)$, $(1,3)$, and $(2,3)$, with product supports
\[
 \{1,3,6,8\},\qquad \{2,3,6,7\},\qquad \{3,4,5,6\}\pmod9.
\]
Kanade's asymptotic calculation leads to
\begin{equation}\label{eq:kanadedilog}
 L(Q_1)+\frac13L(Q_2^3)=\frac{4\pi^2}{27},\qquad
 Q_1=1-2\sin\frac\pi{18},\quad
 Q_2^3=4\sin^2\frac\pi{18}+4\sin\frac\pi{18}.
\end{equation}
This is the model for changing both the cubic field and the symmetrizer.
The present results do not depend on the sum--product formulas in this
setting or on a particular proof of \eqref{eq:kanadedilog}.
The modulus-nine supports overlap in $\{3,6\}$, whereas the cosets
\eqref{eq:cosets} partition the nonzero residues. The latter partition
explains the product-norm identity in \eqref{eq:theta}; it is not a general criterion for
sum--product identities.

The proved modulo-nine evaluations therefore provide a complete
sum--product model for the comparison. The level-$13$ evaluations in
Conjecture~\ref{conj:products} remain open. In both systems the product
modulus and the second denominator step are distinct pieces of data.

\subsection{A four-term identity and its rank-four companion}
\label{sec:AG}
Kirillov's formula \cite[(1.16), with $r=9$ and $j=0$]{Kirillov} gives
\begin{equation}\label{eq:AGdilog}
 \sum_{a=2}^{5}L\!\left(\frac{\sin^2(\pi/11)}{\sin^2(a\pi/11)}\right)
 =\frac8{11}L(1)=\frac{4\pi^2}{33}.
\end{equation}
To identify the associated Nahm system, put
\[
 \vartheta=\frac\pi{11},\qquad
 y_i=\frac{\sin^2\vartheta}{\sin^2((i+1)\vartheta)},\qquad
 z_i=1-y_i\quad(1\leq i\leq4).
\]
Then
\begin{equation}\label{eq:AGmatrix}
 1-z_i=\prod_{j=1}^4z_j^{2\min(i,j)},\qquad
 A_{11}^{AG}=2\begin{pmatrix}1&1&1&1\\1&2&2&2\\1&2&3&3\\1&2&3&4\end{pmatrix},
 \qquad D=I_4.
\end{equation}
The complement equations follow from
$\sin^2((i+1)\vartheta)-\sin^2\vartheta
=\sin(i\vartheta)\sin((i+2)\vartheta)$ and telescoping.
The matrix is positive definite, since its half has entries
$\min(i,j)$ and is the Gram matrix of the initial-segment vectors.

The associated proved series are the five modulus-$11$ Andrews--Gordon
identities \cite{Andrews}. Put $N_j=n_j+\cdots+n_4$ and interpret the
linear term as zero when $a=5$. Then
\begin{equation}\label{eq:AGsum}
 \mathcal A_a(q):=
 \sum_{n_1,\ldots,n_4\geq0}
 \frac{q^{N_1^2+\cdots+N_4^2+N_a+\cdots+N_4}}
 {\prod_{j=1}^4(q;q)_{n_j}}
 =\frac{(q^a,q^{11-a},q^{11};q^{11})_\infty}{(q;q)_\infty},
 \quad 1\leq a\leq5.
\end{equation}
Here $\tfrac12\mathbf n^TA_{11}^{AG}\mathbf n=\sum_jN_j^2$.
Equations \eqref{eq:lambda} and \eqref{eq:AGdilog} give
$\Lambda_{11}^{AG}=\sum_iL(1-z_i)=4\pi^2/33$, consistent with the eight
permitted residues on the product side.

For the series \eqref{eq:AGsum}, put
\[
 \chi_a(\tau)=q^{\kappa_a}\mathcal A_a(q),\qquad
 \kappa_a=\frac{(11-2a)^2}{88}-\frac1{24},\qquad
 \boldsymbol\chi=(\chi_1,\ldots,\chi_5)^T.
\]
In this ordering the classical inversion formula is
\begin{equation}\label{eq:AGtransform}
 \boldsymbol\chi(-1/\tau)=S_{11}^{AG}\boldsymbol\chi(\tau),\qquad
 (S_{11}^{AG})_{ab}=\frac2{\sqrt{11}}(-1)^{a+b}
 \sin\frac{2\pi ab}{11}\quad(1\leq a,b\leq5).
\end{equation}
This is the Andrews--Gordon sine matrix; see also
\cite[Section~4]{Ono} for the general family and conventions.
Its fifth row is positive:
\[
 (S_{11}^{AG})_{5b}=\frac2{\sqrt{11}}\sin\frac{b\pi}{11}.
\]
Consequently the four arguments in \eqref{eq:AGdilog} satisfy
\begin{equation}\label{eq:AGratios}
 y_i=\left(\frac{(S_{11}^{AG})_{51}}{(S_{11}^{AG})_{5,i+1}}\right)^2,
 \qquad 1\leq i\leq4.
\end{equation}
Also $\kappa_5=-1/33$, so the dominant transformed term yields
$\Lambda_{11}^{AG}=4\pi^2/33$. This classical example exhibits a Rogers
identity, its Nahm system and proved products, and a modular matrix whose
row ratios recover its Rogers arguments.

More generally, the $j=0$ specialization of Kirillov's formula at odd
modulus $M\geq5$ has $(M-3)/2$ terms and value $(M-3)L(1)/M$.
The corresponding classical Andrews--Gordon family has rank $(M-3)/2$
and vector dimension $(M-1)/2$. At $M=13$, its five-term value
$10L(1)/13$ accompanies a rank-five, six-component system. Our
rank-two system \eqref{eq:cubicmatrix} has weighted reflected value
$4L(1)/13$, and the associated products have the two-dimensional
Fricke law of Theorem~\ref{thm:main}. The sum--product identification
remains Conjecture~\ref{conj:products}.

The rank-three modulus-$11$ row of Table~\ref{tab:transforms}
is a different Nahm system, despite the similar sine ratios in its
transformation coefficients. Such ratios do not determine the
summation rank or the symmetrizer.

\section{A finite arithmetic source and a restricted three-matrix theorem}
\label{sec:sieve}
The preceding examples used one construction in degrees one, two and
three. We now restrict its source to a specified finite cubic family;
the quadratic obstructions do not enter this classification. The purpose
is to restrict the quadratic data before searching for linear terms.
Three integers must be distinguished:
the coefficient $m$ in the Rogers relation, the conductor $N$ of its
number field when this is abelian, and the modulus $M$ of a proposed
product. These integers need not agree. In particular, the conductor-nine
example has denominator steps $(1,3)$, whereas the coefficient-$31$
identity belongs to the conductor-$19$ field.

\subsection{The scope of the classification input}
The companion manuscript \cite{CubicClassification} studies non-degenerate
integral monomial configurations over cubic fields. In the simplest-cubic
orders, the Vukusic--Ziegler unit-equation theorem \cite{VZ} gives a finite
source of exceptional-unit configurations. The companion manuscript
extracts ten normalized two-term Rogers identities and supplies their
five-term certificates. Its broader totally real cubic reduction also
leaves a second family
\[
 X^3-(s+1)X^2+sX+1,\qquad s\geq4.
\]
Possible nontrivial configurations in that family are confined there to
an explicit finite region, whose emptiness has not been proved.
Thus effective finiteness in that specified cubic monomial class and
completeness of the ten-row simplest-cubic table are distinct assertions.
Neither assertion classifies arbitrary algebraic two-term dilogarithm
equalities or arbitrary modular Nahm sums.

Table~\ref{tab:arithmetic-sieve} records the exact monomial data of
these ten identities in the notation of \eqref{eq:monomial-input}.
The pair $x,1-x$ is multiplicatively independent in each row, so
Proposition~\ref{prop:monomial-matrix} applies whenever $h\ne0$.
We allow simultaneous reflection as described in
Section~\ref{sec:conversion}. Thus the cubic restriction enters through
this finite source, not through the conversion formula.

\begin{table}[htbp]
\caption{The ten normalized simplest-cubic records from
\cite{CubicClassification}. The identity is $mL(x)+L(y)=c_0\pi^2$.
The last column is the formal value $(m+1-6c_0)/m$; it is used as a
radial density only after the positive-definiteness test.}\label{tab:arithmetic-sieve}
\centering\small
\renewcommand{\arraystretch}{1.12}
\begin{tabular}{@{}rrrrrrrrrr@{}}\toprule
Row&$N$&$m$&$c_0$&$p$&$h$&$u$&$v$&$\det(AD)$&$6\Lambda/\pi^2$\\\midrule
1&7&2&$2/7$&$-2$&2&$-1$&2&1&$9/14$\\
2&7&2&$5/21$&2&0&2&$-1$&$\text{--}^{\ast}$&$11/14$\\
3&9&3&$4/9$&$-2$&1&$-1$&2&3&$4/9$\\
4&7&7&$2/3$&$-2$&3&1&2&$-7/3$&$4/7$\\
5&13&13&$5/3$&$-4$&1&$-3$&4&39&$4/13$\\
6&9&19&$4/3$&$-2$&5&3&2&$-57/5$&$12/19$\\
7&7&13&$4/3$&$-4$&3&$-1$&4&$13/3$&$6/13$\\
8&19&31&4&$-6$&1&$-5$&6&155&$8/31$\\
9&7&2&$13/42$&$-1$&1&0&2&0&$4/7$\\
10&7&97&$15/2$&$-8$&11&3&8&$-291/11$&$53/97$\\\bottomrule
\end{tabular}
\par\smallskip
\begin{minipage}{0.96\textwidth}\footnotesize
$^{\ast}$Row 2 has $h=0$, so the conversion and determinant formula
$\det(AD)=-mu/h$ do not apply. Here $y=x^2$ and $x,1-x$ are
multiplicatively independent; this orientation has no rational Nahm
matrix, and its simultaneous reflection is singular.
\end{minipage}
\end{table}
The values of $c_0$ in this table are established in the companion
manuscript; the three-matrix exclusion below uses only the displayed
integer exponent records. The conductor-$13$ evaluation needed for our
principal example is proved independently in Appendix~\ref{app:certificate}.

\subsection{The restricted cubic sieve}
\begin{theorem}[Restricted three-matrix sieve]\label{thm:three-matrix}
Take one of the ten records of Table~\ref{tab:arithmetic-sieve}, or its
simultaneous reflection, and require:
\begin{enumerate}
\item the coefficient $m$ is greater than two;
\item its ordered arguments give a symmetrizable Nahm system with
denominator steps $(1,m)$ and positive cross term;
\item the quadratic matrix is positive definite;
\item its quadratic exponent can be made integer-valued by a rational
linear term, without restricting the summation lattice or replacing $q$.
\end{enumerate}
Then the only quadratic forms are
\begin{equation}\label{eq:three-forms}
\begin{array}{c|c|c|c}
m&A&\tfrac12(r,s)AD(r,s)^T&6\Lambda/\pi^2\\\hline
3&\left(\begin{smallmatrix}2&1\\3&2\end{smallmatrix}\right)
 &r^2+3rs+3s^2&4/9\\[2pt]
13&\left(\begin{smallmatrix}4&1\\13&4\end{smallmatrix}\right)
 &2r^2+13rs+26s^2&4/13\\[2pt]
31&\left(\begin{smallmatrix}6&1\\31&6\end{smallmatrix}\right)
 &3r^2+31rs+93s^2&8/31
\end{array}
\end{equation}
This is a classification of the indicated quadratic data within the
specified ten-record source, not a modularity theorem for their sums.
\end{theorem}
\begin{proof}
Among rows with $m>2$, rows 4, 6 and 10 have negative determinant.
The remaining rows are 3, 5, 7 and 8. Row 7 has off-diagonal entry
$13/3$ in $AD$, so the integrality criterion excludes it. The other
three give \eqref{eq:three-forms} directly. When $AD$ is invertible,
simultaneous reflection replaces $A$ by $A^{-1}$: take logarithms of
the complement equations and interchange the two sides. Positive
definiteness is preserved, and the off-diagonal entry changes sign.
Thus no reflected positive-definite record adds a positive-cross-term
case. The densities follow from \eqref{eq:weightedlambda}.
\end{proof}
If $m=2$ is admitted, row 1 also survives, with
$A=\left(\begin{smallmatrix}1&1/2\\1&1\end{smallmatrix}\right)$,
$D=\operatorname{diag}(1,2)$, and quadratic exponent
$r^2/2+rs+s^2$. This is the conductor-seven example of
\cite[(42)--(46)]{Mizuno}, whose parity components form the
six-dimensional vector discussed in Section~\ref{sec:related}.
The classical three-dimensional modulus-$7$ vector listed in
Table~\ref{tab:transforms} comes instead from $A_7^{AG}$ and $D=I_2$,
as described in Section~\ref{sec:related}. These two
conductor-seven examples have different quadratic forms and denominator
steps.

Row 9 is singular. Row 2 has $y=x^2$, while $x,1-x$ are independent,
so its displayed orientation has no rational Nahm equation for $1-x$;
its reflection gives a singular matrix. Fractional exponents and
restricted summation lattices belong to a larger class and are not
excluded by this theorem.

\subsection{Product density and the special role of level thirteen}
The density argument applies to a positive finite sum of constituents
sharing the same quadratic datum. Each has the same exponential
$e^{\Lambda/\varepsilon}$ and a positive leading amplitude, so their
leading exponential cannot cancel. If such a sum equals $q^\kappa P_S(q)$,
where $P_S$ is a reciprocal product with each supported residue occurring
once, then
\begin{equation}\label{eq:density-sieve}
 \frac{|S|}{M}=\frac{m+1-6c_0}{m}.
\end{equation}
The power $q^\kappa$ does not change this equation.

\begin{corollary}\label{cor:minimal-coset}
For the three rows in \eqref{eq:three-forms}, the product modulus is
respectively a multiple of $9$, $13$, or $31$. At these smallest
density-compatible moduli the support sizes are $4$, $4$, and $8$.
At modulus $31$, a support of size eight cannot be a multiplicative
subgroup coset. Among the prime minimal moduli in this three-row list,
only $13$ permits the proposed subgroup-coset support.
\end{corollary}
\begin{proof}
The three densities are in lowest terms, so their denominators divide
$M$. A subgroup coset in $(\ZZ/31\ZZ)^\times$ has size dividing $30$,
and $8\nmid30$. At $13$, the subgroup of cubic residues has order
$(13-1)/3=4$, as required.
\end{proof}
This does not exclude an arbitrary eight-residue product modulo $31$,
a product at a larger modulus, a quotient with other multiplicities, or
modularity without such a product. It also does not force the three
specific level-$13$ supports: density alone allows fifteen symmetric
four-element supports. Their exact Fricke closure is the additional
input supplied by \cite[Theorem~1.1]{FrickeCompanion}.

There is a documented reason why the index-$13$ example is absent from
one influential search: Mizuno's rank-two search used $1\le d_i\le10$
\cite[Section~4]{Mizuno}. Moreover, that search tests individual Nahm
sums, whereas the present conjectures use two constituents. These are
specific differences in search scope; they do not establish that earlier
methods could never discover the identities or that every remaining
possibility has been exhausted.

The higher asymptotic equations of Section~\ref{sec:higher} provide the
next test for the surviving index-$31$ datum. A rigorously nonzero
coefficient excludes a specified choice of linear terms. To exclude an
unbounded family requires exact elimination of those parameters or
certified bounds covering all of them. A finite numerical search, however
large, supplies only a bounded negative result.
\section{Candidate sums, exact dependencies, and asymptotic consistency}
\label{sec:conjectures}
For the quadratic form obtained in Section~\ref{sec:framework}, define
\begin{equation}\label{eq:Fdef}
 F_{a,b}(q)=\sum_{r,s\geq0}
 \frac{q^{2r^2+13rs+26s^2+ar+bs}}{(q;q)_r(q^{13};q^{13})_s}.
\end{equation}
We compare these sums with the products $P_0,P_1,P_2$ in
\eqref{eq:products}. The index $(1,13)$ records the two denominator steps.
The following series converge absolutely for $|q|<1$.
\begin{conjecture}\label{conj:products}
For $|q|<1$,
\begin{align}
 F_{0,0}+qF_{2,13}&=P_0,\label{eq:conj0}\\
 F_{0,0}+q^6F_{7,26}&=P_1,\label{eq:conj1}\\
 F_{4,13}+q^4F_{6,26}&=P_2.\label{eq:conj2}
\end{align}
\end{conjecture}
For example, \eqref{eq:conj0} is the single displayed expression
\[
 \sum_{r,s\geq0}
 \frac{q^{2r^2+13rs+26s^2}(1+q^{2r+13s+1})}
 {(q;q)_r(q^{13};q^{13})_s}=P_0(q).
\]
The corresponding numerator factors for \eqref{eq:conj1} and
\eqref{eq:conj2} are $1+q^{7r+26s+6}$ and
$q^{4r+13s}(1+q^{2r+13s+4})$. These are assertions about sums of two
constituents; they do not assert that any individual $F_{a,b}$ is modular.

\begin{proposition}\label{prop:dependencies}
Without assuming Conjecture~\ref{conj:products}, one has
\begin{equation}\label{eq:sumdependency}
 F_{2,13}-F_{4,13}=q^4F_{6,26}+q^5F_{7,26}.
\end{equation}
Together with the product relation in \eqref{eq:theta}, this shows
that any two assertions in the conjecture imply the third.
\end{proposition}
\begin{proof}
Write $Q(r,s)=2r^2+13rs+26s^2$. Then
\[
 F_{2,13}-F_{4,13}=
 \sum_{r,s\geq0}\frac{q^{Q(r,s)+2r+13s}(1-q^{2r})}
 {(q;q)_r(q^{13};q^{13})_s}.
\]
The $r=0$ terms vanish. Use
$(1-q^{2r})/(q;q)_r=(1+q^r)/(q;q)_{r-1}$ and replace $r$ by $r+1$.
The new numerator is
$q^{Q(r,s)+6r+26s+4}(1+q^{r+1})$, proving
\eqref{eq:sumdependency}. Equation~\eqref{eq:theta} gives $P_0=P_1+qP_2$.
Multiplying the sum-side relation by $q$ gives the same linear dependence
among the conjectural sum sides as among the products.
\end{proof}

\subsection{Exact finite evidence}\label{sec:finite}
The first 801 coefficients of each pair in
\eqref{eq:conj0}--\eqref{eq:conj2} agree, through $q^{800}$ inclusive.
The comparison uses exact integer arithmetic and direct expansions of both
sides. For all five nonnegative linear-term pairs, a summand contributing
through degree $N$ must satisfy
\[
 2r^2+13rs+26s^2\leq N,\qquad
 r\leq\lfloor\sqrt{N/2}\rfloor,\quad
 s\leq\lfloor\sqrt{N/26}\rfloor.
\]
If $p_r(n)=[q^n](q;q)_r^{-1}$, then
\[
 p_0(n)=\begin{cases}1,&n=0,\\0,&n>0,\end{cases}\qquad
 p_r(n)=p_{r-1}(n)+p_r(n-r),
\]
with zero values at negative indices. The second denominator contributes
$\sum_{j\geq0}p_s(j)q^{13j}$. Convolving these finite arrays and applying
the numerator shifts gives the sum coefficients. Independently, multiplying
by $(1-q^n)^{-1}$ for each permitted $n\leq N$ gives the product
coefficients, using the ascending update $a_d\leftarrow a_d+a_{d-n}$.
These recurrences specify the entire finite computation without assuming a
common recurrence for the conjecturally equal series.

These finite comparisons support the conjectures but do not prove them.
The earlier bounded searches described in the preliminary long draft
also failed to find a single constituent representing one of the products.
No global minimality of the number of summands is asserted. The present
paper separates the finite cubic exclusions of Section~\ref{sec:sieve}
from the individual-sum quadratic obstructions of
Theorem~\ref{thm:quadratic-obstruction}.

\subsection{Exact leading amplitudes}\label{sec:amplitudes}
The saddle calculation also recovers all three Fricke amplitudes,
independently of the finite coefficient comparisons.
\begin{proposition}\label{prop:amplitudes}
For fixed real $a,b$, as $\varepsilon\downarrow0$,
\begin{equation}\label{eq:sumamplitude}
 F_{a,b}(e^{-\varepsilon})\sim
 \kappa\rho^a\sigma^{b/13}e^{\Lambda/\varepsilon},\qquad
 \kappa=\frac{\rho^2(\alpha+1)}{\sqrt{13}},\quad
 \Lambda=\frac{2\pi^2}{39}.
\end{equation}
The three sum sides in \eqref{eq:conj0}--\eqref{eq:conj2} have respective
leading amplitudes $(\alpha,\beta,\gamma)/\sqrt{13}$, agreeing exactly
with the product sides in Proposition~\ref{prop:productasy}.
\end{proposition}
\begin{proof}
Put $\Xi=\operatorname{diag}(\rho/(1-\rho),\sigma/(1-\sigma))$.
The negative Hessian of the action in Section~\ref{sec:radial-action}
is $H=D^{-1}(A+\Xi)$. The Gaussian prefactor in the Nahm-sum
expansion \cite[Section~2]{Mizuno} is
\[
 \kappa=\frac{1}{\sqrt{13\det(H)(1-\rho)(1-\sigma)}}
 =\frac{1}{\sqrt{3+\rho+\sigma-4\rho\sigma}}.
\]
The linear term contributes $\rho^a\sigma^{b/13}$. With
$\alpha=7\rho-2\rho^2$ and
$\sigma=41-66\rho+15\rho^2$, reduction modulo $f$ gives
\[
 (3+\rho+\sigma-4\rho\sigma)\rho^4(\alpha+1)^2=13.
\]
All factors specifying the positive square root are positive, which
proves \eqref{eq:sumamplitude}.

Since $q^k\to1$, the three sum-side amplitudes are
\[
 \kappa\bigl(1+\rho^2\sigma,\;
 1+\rho^7\sigma^2,\;
 \rho^4\sigma+\rho^6\sigma^2\bigr).
\]
Using $\rho^4\sigma=1-\rho$, their identification with the Fricke
amplitudes reduces to
\begin{align*}
 (\alpha+1)(1-\rho+\rho^2)&=\alpha,\\
 (\alpha+1)\rho(1-\rho+\rho^2)&=\beta,\\
 (\alpha+1)(1-\rho)(1-\rho+\rho^2)&=\gamma.
\end{align*}
The first difference is exactly $(1-2\rho)f(\rho)=0$; the other two
follow from the matrix ratios \eqref{eq:matrixratios}.
\end{proof}
Thus both the exponential growth and the leading constants agree.
This is an exact asymptotic consistency result; it does not establish
the infinite sum--product identities.

\section{Mixed signatures and the limits of the arithmetic restriction}
\label{sec:signature}
The field signature and the denominator steps are different data. A
number field of signature $(1,n)$ has degree $1+2n$, whereas an index
$(1,m)$ specifies the denominator factors $(q;q)_r(q^m;q^m)_s$.
A quartic field with a real embedding has signature $(4,0)$ or $(2,1)$,
so it is not of signature $(1,n)$.

\subsection{A positive Nahm family from Lewin's complex cubic identity}
Let $\eta\in(0,1)$ satisfy $\eta^3+\eta^2=1$. Its field has signature
$(1,1)$, and Lewin's identity is
\begin{equation}\label{eq:lewin-synthesis}
 2L(\eta)+L(\eta^2)=\frac{\pi^2}{3}.
\end{equation}
For completeness, the five-term relation at $x=y=\eta$ gives
$2L(\eta)=L(\eta^2)+2L(\eta/(1+\eta))$; use
$\eta/(1+\eta)=\eta^3=1-\eta^2$ and reflection. Also
$1-\eta=\eta^5$ and $1-\eta^2=\eta^3$.

\begin{proposition}\label{prop:lewin-family}
For every rational $b$ with $0<b<15/16$, set
\[
 A_b=\begin{pmatrix}5-2b&b\\2b&3/2-b\end{pmatrix},\qquad
 D=\operatorname{diag}(1,2).
\]
Then $A_bD$ is positive definite with positive cross term, its positive
Nahm solution is $(\eta,\eta^2)$, and its radial action is $\pi^2/12$.
In particular, $b=1/2$ gives the convergent family
\begin{equation}\label{eq:lewin-sums}
 \sum_{r,s\geq0}
 \frac{q^{2r^2+rs+s^2+\ell_1r+\ell_2s}}
 {(q;q)_r(q^2;q^2)_s}.
\end{equation}
The proposition asserts no modularity or product evaluation for these sums.
\end{proposition}
\begin{proof}
The two Nahm equations reduce to
$1-\eta=\eta^{5-2b}(\eta^2)^b=\eta^5$ and
$1-\eta^2=\eta^{2b}(\eta^2)^{3/2-b}=\eta^3$.
The leading entry of $A_bD$ is $5-2b>0$, and its determinant is
$15-16b>0$. Positive definiteness gives uniqueness of the positive
solution and convergence for every fixed rational linear term.
Equation~\eqref{eq:weightedlambda}, with $m=2$ and $c_0=1/3$,
gives $\Lambda=\pi^2/12$.
\end{proof}
The multiplicative dependence of $\eta$ and $\eta^2$ explains why a
single Rogers identity here allows an entire rational family of matrices.
The uniqueness clause of Proposition~\ref{prop:monomial-matrix} explicitly
requires independence. Thus finiteness of a list of dilogarithm values,
without hypotheses on its arguments, need not imply finiteness of matrices.

\subsection{What the nonreal embeddings actually require}
For a single modular symmetrizable Nahm sum, Mizuno's theorem
\cite[Theorem~3.3]{Mizuno} makes the associated weighted Bloch element
torsion. For index $(1,m)$, after clearing the weight denominator, that
element is
\[
 \xi=m[x]+[y].
\]
Consequently every complex embedding $\sigma$ must satisfy
\begin{equation}\label{eq:regulator-test}
 m\,\mathscr D(\sigma x)+\mathscr D(\sigma y)=0,
\end{equation}
where $\mathscr D$ is the Bloch--Wigner dilogarithm. For a number field
the rationalized Bloch group has rank equal to the number of complex
places; its real regulator is injective after tensoring with $\RR$
\cite{ZagierDilog}. In a totally real field the rationalized group
vanishes. In a mixed-signature field the regulator equations are
additional restrictions. Nonreal conjugates do not alone contradict them:
the certificate for \eqref{eq:lewin-synthesis} gives a torsion relation
even in the complex cubic field.

The companion cubic classification \cite{CubicClassification} excludes
five-term certified relations with $m>2$ between exceptional units in
signature $(1,1)$. It does not establish that arbitrary saddle arguments
in mixed-signature fields must be exceptional units, nor does it classify
degree-five and higher fields of signature $(1,n)$. Known quartic
two-term evaluations with other coefficient patterns also prevent a
blanket degree restriction; examples occur in Bytsko's discussion
\cite[Section~3]{Bytsko}. The question for a new higher-degree identity
with the specific positive weighting $m:1$, $m>2$, therefore remains
separate from the cubic theorem.

Finally, a product identity for $F_1+q^kF_2$ does not imply modularity of
either constituent. We do not invoke the single-sum torsion theorem for
such a combination without an additional argument. The common-saddle
density condition \eqref{eq:density-sieve} remains necessary for positive
combinations, as do the higher asymptotic conditions applied to the
combination itself. A classification of these combinations would need
to justify any further Bloch-group implication explicitly.
\appendix
\section{A five-term proof of Theorem~\ref{thm:dilog}}
\label{app:certificate}\label{sec:certificate}
We give the full certificate for Theorem~\ref{thm:dilog}. Euler reflection
and the real Rogers five-term identity are
\begin{align}
 L(x)+L(1-x)&=L(1),\label{eq:reflection}\\
 L(x)+L(y)&=L(xy)+L\!\left(\frac{x(1-y)}{1-xy}\right)
 +L\!\left(\frac{y(1-x)}{1-xy}\right),\quad 0<x,y<1.
 \label{eq:fiveterm}
\end{align}
For a formal symbol $[x]$, put
\begin{align*}
 \FT(x,y)&=[x]+[y]-[xy]
 -\left[\frac{x(1-y)}{1-xy}\right]
 -\left[\frac{y(1-x)}{1-xy}\right],\\
 R(x)&=[x]+[1-x]-[C],
\end{align*}
where $[C]$ is a distinguished formal constant sent to $L(1)$.
Recall $w=1-\rho$, and set $u_1=1+\rho$ and $u_2=(1+w)/w$.
The arrays in Tables~\ref{tab:five} and~\ref{tab:reflection} are interpreted
in $K$, using the real embedding \eqref{eq:rhointerval}.

\begin{table}[H]
\begin{minipage}[t]{0.57\textwidth}
\caption{Five-term arrays.}\label{tab:five}
\centering\small
\renewcommand{\arraystretch}{1.18}
\begin{tabular}{@{}rrcc@{}}\toprule
$i$&$c_i$&$x_i$&$y_i$\\\midrule
1&$-2$&$w$&$w$\\
2&$2$&$u_1^{-1}$&$wu_1$\\
3&$2$&$\rho^{-1}u_2^{-1}$&$w^2u_2$\\
4&$-1$&$u_1^{-1}$&$u_2^{-1}$\\
5&$1$&$u_2/5$&$\rho$\\
6&$-1$&$u_2^{-1}$&$u_2^{-1}$\\
7&$-1$&$wu_1^{-1}u_2$&$wu_2^{-1}$\\
8&$2$&$\rho u_1^{-1}$&$\rho u_1^{-1}$\\
9&$1$&$\rho^{-1}u_2^{-1}$&$\rho^{-1}w^2$\\
10&$1$&$\rho^{-1}w$&$\rho wu_1^{-1}$\\
11&$1$&$\rho^{-1}u_2^{-1}$&$\rho u_2^{-1}$\\
12&$1$&$\rho^{-1}wu_1$&$\rho u_2/5$\\
13&$-1$&$u_1^{-1}u_2^{-1}$&$wu_1$\\
14&$1$&$\rho^{-4}w$&$\rho^2$\\\bottomrule
\end{tabular}
\end{minipage}\hfill
\begin{minipage}[t]{0.40\textwidth}
\caption{Reflection arrays.}\label{tab:reflection}
\centering\small
\renewcommand{\arraystretch}{1.18}
\begin{tabular}{@{}rrc@{}}\toprule
$i$&$d_i$&$z_i$\\\midrule
1&$1$&$\rho^{-3}u_1u_2^{-1}$\\
2&$-1$&$\rho^{-1}wu_1$\\
3&$1$&$w^{-1}u_2^{-1}$\\
4&$-1$&$u_1^{-1}$\\
5&$-1$&$u_2/5$\\
6&$1$&$wu_1^{-1}u_2$\\
7&$11$&$w$\\
8&$-1$&$wu_1$\\\bottomrule
\end{tabular}
\end{minipage}
\end{table}

\begin{proof}[Proof of Theorem~\ref{thm:dilog}]
Expanding the displayed arrays and reducing their arguments modulo
$t^3-4t^2+t+1$ gives the formal identity
\begin{equation}\label{eq:certificate}
 \sum_{i=1}^{14}c_i\FT(x_i,y_i)+\sum_{i=1}^{8}d_iR(z_i)
 =13[\rho]+[\sigma]-10[C].
\end{equation}
This is equality of formal linear combinations after equal field elements
are identified. Each argument has a unique representative
$a+b\rho+c\rho^2$ with rational coefficients. The rules
$\rho^3=4\rho^2-\rho-1$ and \eqref{eq:reduction} suffice to verify every
cancellation, leaving just the three symbols on the right.

All 36 primary arguments in the two tables are in $(0,1)$.
More explicitly, evaluate each reduced quadratic representative on the
rational interval \eqref{eq:rhointerval} by interval arithmetic. Every
resulting interval is contained in $(58/1000,986/1000)$.
Consequently \eqref{eq:reflection} and \eqref{eq:fiveterm} apply to every
array without analytic continuation. Sending $[x]$ to $L(x)$ and $[C]$ to
$\pi^2/6$ in \eqref{eq:certificate} proves the theorem.
\end{proof}
The argument $u_2/5$ is intentional: it equals
$(3+3\rho-\rho^2)/5$, has value $0.9302186817\ldots$, and satisfies
$25X^3-35X^2+12X-1=0$. The certificate is specified entirely by the tables
and the cubic reduction rule.

\section{Exact first-correction obstructions for the quadratic systems}
\label{app:quadratic-obstruction}
For this appendix, let $F_{A,m;a,b}$ denote the sum in
Theorem~\ref{thm:quadratic-obstruction} before multiplication by $q^c$.
The symmetrizable Nahm-sum expansion \cite[Section~2]{Mizuno}, with
the modularity conditions of \cite{VlasenkoZwegers}, gives
\begin{equation}\label{eq:qob-expansion}
 F_{A,m;a,b}(e^{-\varepsilon})
 =K e^{\Lambda/\varepsilon}
 \bigl(1+\beta(a,b)\varepsilon+O(\varepsilon^2)\bigr),\qquad K>0.
\end{equation}
If $q^cF$ is modular of weight zero, its Fourier expansion at the cusp
zero has only exponentially small corrections to its leading term.
Consequently \eqref{eq:qob-expansion} requires $\beta(a,b)=c\in\QQ$.

Exact evaluation of the Gaussian formula in Appendix~\ref{app:quad-formula}
gives the following irrational parts; the omitted part of each $\beta$
is rational for rational $a,b$:
\begin{align*}
 \beta_{A_8}&\in\QQ+\frac{\sqrt3}{72}P_8,&
 \beta_{A_8^{-1}}&\in\QQ+\frac{\sqrt3}{72}P_8^*,\\
 \beta_{A_6}&\in\QQ+\frac{\sqrt2}{8232}P_6,&
 \beta_{A_6^{-1}}&\in\QQ+\frac{\sqrt2}{8232}P_6^*.
\end{align*}
Here the four polynomials are
\begin{align}
 P_8&=7+6a-b-14a^2+10ab-2b^2,\label{eq:qob-p8}\\
 P_8^*&=-7-16a-3b+32a^2+8ab+2b^2,\label{eq:qob-p8dual}\\
 P_6&=820-1680a+336b+2940a^2-1176ab+49b^2,\label{eq:qob-p6a}\\
 P_6^*&=-820+1512a+168b-1764a^2+294ab+245b^2.\label{eq:qob-p6dual}
\end{align}
\begin{proof}[Proof of Theorem~\ref{thm:quadratic-obstruction}]
For $A_8$, rationality of $\beta$ requires $P_8=0$, equivalently
\[
 (6a-7)^2+3(4b-10a+1)^2=220.
\]
For $A_8^{-1}$ the corresponding equation is
\[
 (24a-5)^2+3(8a+4b-3)^2=220.
\]
Neither has a rational solution. Indeed, clearing denominators would give
coprime integers $X,Y,Z$ with $X^2+3Y^2=220Z^2$. Modulo five,
$-3$ is a nonsquare, so $5\mid X,Y$. The equation then forces $5\mid Z$,
contradicting coprimality.

For $A_6$, $P_6=0$ is equivalent to
\[
 (7b-84a+24)^2-84(7a-2)^2=-580,
\]
or $X^2-21Y^2=-145$ with rational $X,Y$. For $A_6^{-1}$, $P_6^*=0$
is equivalent to
\[
 (35b+21a+12)^2-21(21a-8)^2=2900.
\]
Neither norm equation has a rational solution. Clearing denominators
gives $X^2-21Y^2=NZ^2$ with $N=-145$ or $2900$, both congruent to two
modulo three. Reduction modulo three forces $3\mid X,Z$; the equation
then forces $3\mid Y$, again contradicting coprimality.
Thus $\beta$ is irrational for every rational pair $(a,b)$ in all four
cases, which excludes the required normalization.
\end{proof}

\subsection{The first correction from the saddle}\label{app:quad-formula}
This section specifies the finite algebra behind
\eqref{eq:qob-p8}--\eqref{eq:qob-p6dual}. Let $d=(1,m)$, let $z$ be the positive
saddle, and put
\[
 w_i=\frac{z_i}{1-z_i},\qquad
 H=D^{-1}\bigl(A+\operatorname{diag}(w_1,w_2)\bigr),\qquad C=H^{-1}.
\]
The matrix $H$ is symmetric positive definite. Define
\begin{align*}
 g_i&=-\frac{b_i}{d_i}-\frac{w_i}{2},&
 h_i&=\frac{w_i(1+w_i)}2,\\
 T_i&=\frac{w_i(1+w_i)}{d_i},&
 V_i&=-\frac{w_i(1+w_i)(1+2w_i)}{d_i},\qquad (b_1,b_2)=(a,b),\\
 c_0&=-\sum_i d_i\left(\frac1{24}+\frac{w_i}{12}\right).
\end{align*}
For a centered Gaussian vector $U$ of covariance $C$, the coefficient is
\begin{equation}\label{eq:qob-beta}
 \beta=c_0+\mathbb E\left[
 \frac12\sum_i h_iU_i^2+\frac1{24}\sum_i V_iU_i^4
 +\frac12\left(\sum_i g_iU_i+\frac16\sum_i T_iU_i^3\right)^2
 \right].
\end{equation}
Only the moments
\[
 \mathbb E(U_iU_j)=C_{ij},\quad \mathbb E(U_i^4)=3C_{ii}^2,\quad
 \mathbb E(U_iU_j^3)=3C_{ij}C_{jj},
\]
and
\[
 \mathbb E(U_i^3U_j^3)=9C_{ii}C_{jj}C_{ij}+6C_{ij}^3
\]
are needed. Thus \eqref{eq:qob-beta} is an explicit quadratic polynomial in
$a,b$, evaluated by rational arithmetic in the indicated quadratic field.
Use
\[
 z^{(8)}=\left(\sqrt3-1,\frac{2+\sqrt3}{4}\right),\qquad
 z^{(6)}=\left(2-\sqrt2,\frac{2+\sqrt2}{4}\right),
\]
and their complements for the inverse matrices.

For the action $\mathcal S$ of Section~\ref{sec:radial-action}, the
amplitude is
$\exp(-\sum_i b_it_i/d_i)\prod_i(1-e^{-t_i})^{-1/2}$, and the
order-$\varepsilon$ term in the logarithm of the summand is $c_0$.
At the saddle, the negative Hessian is $H$, the pure third and fourth
phase derivatives are $T_i,V_i$, and the first two logarithmic amplitude
derivatives are $g_i,h_i$. Substituting
$t=t^{(0)}+\sqrt\varepsilon\,U$ and integrating the terms of order
$\varepsilon$ gives \eqref{eq:qob-beta}. The discrete Laplace expansion
is justified by \cite[Section~2]{Mizuno}.

\begin{samepage}
\section*{Statements and Declarations}
No empirical datasets were used. The finite certificate and the exact
recurrences and truncation bounds for the coefficient comparisons are given
in the paper. The underlying Fricke work \cite{FrickeCompanion} used
Claude Fable~5.1 (Anthropic) and ChatGPT with the GPT-5.6 Sol model
(OpenAI), as disclosed there. Additional ChatGPT assistance was used for
the present paper's algebraic calculations, verification, literature
comparison, and revisions. Responsibility for the mathematical statements,
proofs, computations, and references remains with the author. The author
received no external funding and declares no conflict of interest.\par
\end{samepage}

\end{document}